\documentclass[11pt]{nyjm}

\newtheorem{theorem}{Theorem}[section]
\newtheorem{thm}{Theorem}
\newtheorem{lemma}[theorem]{Lemma}

\newtheorem{proposition}[theorem]{Proposition}
\newtheorem{conjecture}{Conjecture}

\theoremstyle{definition}

\newcommand{\eqM}{\overset{M}{=}}

\def\BC{\mathbb C}

\def\BZ{\mathbb Z}

\def\BJ{\mathbb J}

\def\CA{\mathcal A}

\def\CR{\mathcal R}

\def\CT{\mathcal T}

\def\ve{\varepsilon}
\def\be { \begin{equation} }
\def\ee { \end{equation} }

\begin{document}

\title[AJ conjecture]{On the AJ conjecture for cables of the figure eight knot}

\author[Anh T. Tran]{Anh T. Tran}
\address{Department of Mathematical Sciences, The University of Texas at Dallas, 
800 W Campbell Rd, FO 35, Richardson TX 75080, USA}
\email{att140830@utdallas.edu}

\begin{abstract}
The AJ conjecture relates the A-polynomial and the colored Jones polynomial of a knot in the 3-sphere. It has been verified for some classes of knots, including all torus knots, most double twist knots, $(-2,3,6n \pm 1)$-pretzel knots, and most cabled knots over torus knots. In this paper we study the AJ conjecture for $(r,2)$-cables of a knot, where $r$ is an odd integer. In particular, we show that the $(r,2)$-cable of the figure eight knot satisfies the AJ conjecture if $r$ is an odd integer satisfying $|r| \ge 9$.
\end{abstract}

\thanks{2010 {\em Mathematics Classification:} Primary 57N10. Secondary 57M25.}

\maketitle

\section{Introduction}

\subsection{The colored Jones function} For a knot $K$ in the 3-sphere and a positive integer $n$, let $J_K(n) \in \BZ[t^{\pm 1}]$ denote the $n$-colored Jones polynomial of $K$ with framing zero. The polynomial $J_K(n)$ is the quantum link invariant, as defined by Reshetikhin and Turaev \cite{RT}, associated to the Lie algebra $sl_2(\BC)$, with the color $n$ standing for the irreducible $sl_2(\BC)$-module $V_{n}$ of dimension $n$. Here we use the functorial normalization, i.e. the one for which the colored Jones polynomial of the unknot $U$ is
$$J_U(n)=[n] := \frac{t^{2n}- t^{-2n}}{t^2 -t^{-2}}.$$ 
For example, the colored Jones polynomial of the figure eight knot $E$ is  
$$J_E(n)= [n] \sum_{k=0}^{n-1} \prod_{l=1}^k 
(t^{4n}+t^{-4n}-t^{4l}-t^{-4l}).$$ 

It is known that $J_K(1)=1$ and $J_K(2)$ is the usual Jones polynomial \cite{Jones}. The colored Jones polynomials of higher colors are more or less the usual Jones polynomials of parallels of the knot. The color $n$ can be assumed to take negative integer values by setting $J_K(-n) = - J_K(n)$. In particular, we have $J_K(0)=0$.

The colored Jones polynomials are not random. For a fixed knot $K$, Garoufalidis and Le \cite{GL} proved that the colored Jones function $J_K: \BZ \to \BZ[t^{\pm 1}]$ satisfies a non-trivial linear recurrence relation of the form $$\sum_{k=0}^da_k(t,t^{2n})J_K(n+k)=0,$$
where $a_k(u,v) \in \BC[u,v]$ are polynomials with greatest common divisor 1. 

\subsection{Recurrence relations and $q$-holonomicity} Let $\CR:=\BC[t^{\pm 1}]$. Consider a discrete function $f: \BZ \to \CR$, and define the linear operators $L$ and $M$ acting on such functions by
$$(Lf)(n) := f(n+1), \qquad (Mf )(n) := t^{2n} f(n).$$
It is easy to see that $LM = t^2 ML$. The inverse operators $L^{-1}, M^{-1}$ are well-defined. We can consider $L,M$ as elements of the quantum torus
$$ \mathcal T := \mathbb \CR\langle L^{\pm1}, M^{\pm 1} \rangle/ (LM - t^2 ML),$$
which is a non-commutative ring.

The recurrence ideal of the discrete function $f$ is the left ideal $\CA_f$ in $\CT$ that annihilates $f$:
$$\CA_f:=\{P \in \CT \mid Pf=0\}.$$
We say that $f$ is $q$-holonomic, or $f$ satisfies a non-trivial linear recurrence relation, if $\CA_f \not= 0$. For example, for a fixed knot $K$ the colored Jones function $J_K$ is $q$-holonomic.

\subsection{The recurrence polynomial of a $q$-holonomic function} Suppose that $f: \BZ \to \CR$ is a $q$-holonomic function. Then $\CA_f$ is a non-zero left ideal of $\CT$. The ring $\mathcal T$ is not a principal left ideal domain, i.e. not every left ideal of $\mathcal T$ is generated by one element. Garoufalidis \cite{Ga04} noticed that by adding all inverses of polynomials in $t,M$ to $\mathcal T$ we get a principal left ideal domain $\tilde\CT$, and hence from the ideal $\CA_K$ we can define a polynomial invariant. Formally, we can proceed as follows. Let $\CR(M)$ be the fractional field of the polynomial ring $\CR[M]$. Let $\tilde{\CT}$ be the  set of all Laurent polynomials in the variable $L$ with coefficients in $\CR(M)$:
$$\tilde{\CT} =\left\{\sum_{k \in \BZ} a_k(M)L^k \mid a_k(M) \in \CR(M),~a_k=0\text{~almost always}\right\},$$
and define the product in $\tilde{\CT}$ by $a(M)L^{k} \cdot b(M)L^{l}=a(M)b(t^{2k}M)L^{k+l}$.  

Then it is known that every left ideal in $\tilde{\CT}$ is principal, and $\CT$ embeds as a subring of $\tilde{\CT}$. The extension $\tilde\CA_f:=\tilde\CT\mathcal A_f$ of $\CA_f$ in $\tilde{\CT}$ is then generated by a single polynomial 
$$\alpha_f(t,M,L) = \sum_{k=0}^{d} \alpha_{f,k}(t,M) \, L^k,$$
where the degree in $L$ is assumed to be minimal and all the
coefficients $\alpha_{f,k}(t,M)\in \BC[t^{\pm1},M]$ are assumed to
be co-prime. The polynomial $\alpha_f$ is defined up to a polynomial in $\mathbb C[t^{\pm 1},M]$. We call $\alpha_f$ the recurrence polynomial of the discrete function $f$.

When $f$ is the colored Jones function $J_K$ of a knot $K$, we let $\CA_K$ and $\alpha_K$ denote the recurrence ideal $\CA_{J_K}$ and the recurrence polynomial $\alpha_{J_K}$ of $J_K$ respectively. We also say that $\CA_K$ and $\alpha_{K}$ are  the recurrence ideal and the recurrence polynomial of the knot $K$. Since $J_K(n) \in \BZ[t^{\pm 1}]$, we can assume that $\alpha_K(t,M,L)=\sum_{k=0}^d \alpha_{K,k}(t,M)L^k$
where all the coefficients $\alpha_{K,k} \in \BZ[t^{\pm 1}, M]$ are co-prime.

\subsection{The AJ conjecture}

The colored Jones polynomials are powerful invariants of knots, but little is known about their relationship with classical topology invariants like the fundamental group. Inspired by the theory of noncommutative A-ideals of Frohman, Gelca and Lofaro \cite{FGL, Ge} and the theory of $q$-holonomicity of quantum invariants of Garoufalidis and Le \cite{GL}, Garoufalidis \cite{Ga04} formulated the following conjecture that relates the A-polynomial and the colored Jones polynomial of a knot in the 3-sphere.

\begin{conjecture}
\label{c1}
{\bf (AJ conjecture)} For every knot $K$, $\alpha_K |_{t=-1}$ is equal to the $A$-polynomial, up to a factor depending on $M$ only.
\end{conjecture}

The A-polynomial of a knot was introduced by Cooper et al. \cite{CCGLS}; it describes the $SL_2(\BC)$-character variety of the knot complement as viewed from the boundary torus. The A-polynomial carries important information about the geometry and topology of the knot. For example, it distinguishes the unknot from other knots \cite{DG, BZ}, and the sides of its Newton polygon give rise to incompressible surfaces in the knot complement \cite{CCGLS}. Here in the definition of the $A$-polynomial, we also allow the factor $L-1$ coming from the abelian component of the character variety of the knot group. Hence the A-polynomial in this paper is equal to $L-1$ times the A-polynomial defined in \cite{CCGLS}.

The AJ conjecture has been verified for the trefoil knot, the figure eight knot (by Garoufalidis \cite{Ga04}), all torus knots (by Hikami \cite{Hi}, Tran \cite{Tr}), some classes of two-bridge knots and pretzel knots including most double twist knots and $(-2,3,6n \pm 1)$-pretzel knots (by Le \cite{Le06}, Le and Tran \cite{LTaj}), the knot $7_4$  (by Garoufalidis and Koutschan \cite{GK}), and most cabled knots over torus knots (by Ruppe and Zhang \cite{RZ}). 

Note that there is a stronger version of the AJ conjecture, formulated by Sikora \cite{Si}, which relates the recurrence ideal and the A-ideal of a knot. The A-ideal determines the A-polynomial of a knot. This conjecture has been verified for the trefoil knot (by Sikora \cite{Si}), all torus knots \cite{Tr} and most cabled knots over torus knots \cite{Tr_strongAJ}.

\subsection{Main result} Suppose $K$ is a knot with framing zero, and $r,s$ are two integers with $c$ their greatest common divisor. The $(r,s)$-cable $K^{(r,s)}$ of $K$ is the
link consisting of $c$ parallel copies of the $(\frac{r}{c},\frac{s}{c})$-curve on the torus boundary of a tubular neighborhood of $K$. Here an $(\frac{r}{c},\frac{s}{c})$-curve is a curve that is homologically
equal to $\frac{r}{c}$ times the meridian and $\frac{s}{c}$ times the longitude on the torus boundary.
The cable $K^{(r,s)}$ inherits an orientation from $K$, and we assume that each component of $K^{(r,s)}$ has framing zero. Note that if $r$ and $s$ are co-prime, then $K^{(r,s)}$ is again a knot. 

In \cite{LTvol}, we studied the volume conjecture \cite{Ka, MuM} for $(r,2)$-cables of a knot and especially $(r,2)$-cables of the figure eight knot, where $r$ is an integer. 
In this paper we study the AJ conjecture for $(r,2)$-cables of a knot, where $r$ is an odd integer. In particular, we will show the following.

\begin{thm}
\label{main}
The $(r,2)$-cable of the figure eight knot satisfies the AJ conjecture if $r$ is an odd integer satisfying $|r| \ge 9$.
\end{thm}

\subsection{Plan of the paper} In Section \ref{cable} we prove some properties of the colored Jones polynomial of cables of a knot. In Section \ref{figure 8} we study the AJ conjecture for $(r,2)$-cables of the figure eight knot and prove Theorem \ref{main}.

\subsection{Acknowledgment} I would like to thank Thang T.Q. Le and Xingru Zhang for helpful discussions. I would also like to thank the referee for comments and suggestions. Dennis Ruppe \cite{Ru} has independently obtained a similar result to Theorem \ref{main}.

\section{The colored Jones polynomial of cables of a knot} 

\label{cable}

Recall from the introduction that for each positive integer $n$, there is a unique irreducible $sl_2(\BC)$-module $V_n$ of dimension $n$.  

From now on we assume that $r$ is an odd integer. Then the $(r,2)$-cable $K^{(r,2)}$ of a knot $K$ is a knot. The calculation of the colored Jones polynomial of $K^{(r,2)}$ is standard: we decompose $V_n \otimes V_n$ into irreducible components
$$V_n \otimes V_n = \bigoplus_{k=1}^n V_{2k-1}.$$
Since the $R$-matrix commutes with the actions of the quantized algebra, it acts on each component $V_{2k-1}$ as a scalar $\mu_k$ times the identity. The value of $\mu_k$ is well-known:
$$\mu_k = (-1)^{n-k} t^{-2(n^2-1)} t^{2k(k-1)}.$$
Hence from the theory of quantum invariants (see  e.g. \cite{Ohtsuki}), we have
\begin{eqnarray}
J_{K^{(r,2)}}(n) &=& \sum_{k=1}^n \mu_k^r J_K(2k-1) \nonumber\\
                 &=& t^{-2r(n^2-1)}\sum_{k=1}^n (-1)^{r(n-k)} t^{2rk(k-1)} J_K(2k-1). \label{cables}
\end{eqnarray}

Note that $t$ in this paper is equal to $q^{1/4}$ in \cite{LTvol}.
\begin{lemma}
\label{J_C}
We have
$$J_{K^{(r,2)}}(n+1) = - t^{-2r(2n+1)}J_{K^{(r,2)}}(n)+ t^{-2rn}J_K(2n+1).$$
\end{lemma}

\begin{proof}
From Eq. \eqref{cables} we have 
\begin{eqnarray*}
&& J_{K^{(r,2)}}(n+1)\\
 &=& t^{-2r(n^2+2n)}\sum_{k=1}^{n+1} (-1)^{r(n+1-k)} t^{2rk(k-1)} J_K(2k-1)\\
                   &=& t^{-2rn}J_K(2n+1) + (-1)^r t^{-2r(n^2+2n)}\sum_{k=1}^{n} (-1)^{r(n-k)} t^{2rk(k-1)} J_K(2k-1)\\
                   &=& t^{-2rn}J_K(2n+1) + (-1)^r t^{-2r(2n+1)}J_{K^{(r,2)}}(n).
\end{eqnarray*}
The lemma follows, since $(-1)^r=-1$.
\end{proof}

Let $\BJ_K(n):=J_K(2n+1)$. Note that $q$-holonomicity is preserved under taking subsequences of the form $kn+l$, see e.g. \cite{KK}. Since $J_K$ is $q$-holonomic, we have the following.

\begin{proposition}
\label{BJ}
For a fixed knot $K$, the function $\BJ_K$ is $q$-holonomic.
\end{proposition}

Note that $\BJ_K (n-1)+\BJ_K (-n)=0$. Recall that $\CA_{\BJ_K}$ and $\alpha_{\BJ_K}$ denote the recurrence ideal and the recurrence polynomial of $\BJ_K$ respectively.

\begin{lemma}
\label{inv}
If $P(t,M,L) \in \CA_{\BJ_K}$ then $P(t,(t^{2}M)^{-1},L^{-1}) \in \CA_{\BJ_K}$.
\end{lemma}

\begin{proof}
Suppose that $P(t,M,L)=\sum \lambda_{k,l} M^kL^l$, where $\lambda_{k,l} \in \CR=\BC[t^{\pm 1}]$, annihilates $\BJ_K$. Since $\BJ_K (n-1)+\BJ_K (-n)=0$ for all integers $n$, we have
\begin {eqnarray*} 
0 &=& P\BJ_K(-n-1) \\
  &=& \sum \lambda_{k,l} \, t^{-2(n+1)k} \BJ_K(-n-1+l) \\
  &=& - \sum \lambda_{k,l} \, t^{-2(n+1)k} \BJ_K(n-l) \\
  &=& - \sum \lambda_{k,l} (t^{2}M)^{-k} L^{-l} \BJ_K(n).
\end{eqnarray*}
Hence $P(t,(t^2M)^{-1},L^{-1}) \BJ_K=0$.
\end{proof}

For a Laurent polynomial $f(t) \in \CR$, let $d_+[f]$ and $d_-[f]$ be respectively the maximal and minimal degree of $t$ in $f$. The difference $br[f]:=d_+[f]-d_-[f]$ is called the breadth of $f$.

\begin{lemma}
\label{deg}
Suppose $K$ is a non-trivial alternating knot. Then $br[\BJ_K(n)]$ is a quadratic polynomial in $n$.
\end{lemma}

\begin{proof}
Since $K$ is a non-trivial alternating knot, \cite[Proposition 2.1]{Le06} implies that $br[J_K(n)]$ is a quadratic polynomial in $n$. Since $br[\BJ_K(n)]= br[J_K(2n+1)]$, the lemma follows.
\end{proof}

\begin{proposition}
\label{alt}
Suppose $K$ is a non-trivial alternating knot. Then the recurrence polynomial $\alpha_{\BJ_K}$ of $\BJ_K$ has $L$-degree $>1$.
\end{proposition}

\begin{proof}

Suppose that $\alpha_{\BJ_K}(t,M,L)=P_1(t,M)L + P_0(t,M)$, where $P_1, P_0 \in \BZ[t^{\pm 1}, M]$ are co-prime. Note that the polynomial $\alpha_{\BJ_K}(t,(t^2M)^{-1},L^{-1})=P_1(t,t^{-2}M^{-1})L^{-1} + P_0(t,t^{-2}M^{-1})$ is in the recurrence ideal $\CA_{\BJ_K}$ of $\BJ_K$, by Lemma \ref{inv}. Since $\alpha_{\BJ_K}$ is the generator of $\tilde{\CA}_{\BJ_K}=\tilde{\CT} \CA_{\BJ_K}$ in $\tilde{\CT}$, there exists $\gamma(t,M) \in \CR(M)$ such that
$$\gamma(t,M) L \left( P_1(t,t^{-2}M^{-1})L^{-1} + P_0(t,t^{-2}M^{-1}) \right)= P_1(t,M)L + P_0(t,M).$$ 
This is equivalent to $P_0(t,M)=\gamma(t,M) P_1(t,t^{-4}M^{-1})$ and $P_1(t,M)=\gamma(t,M) P_0(t,t^{-4}M^{-1}).$ Since $P_0$ and $P_1$ are coprime in $\BZ[t^{\pm 1}, M]$, it follows from the above equations that $\gamma(t,M)$ is a unit element in $\BZ[t^{\pm 1}, M^{\pm 1}]$, i.e. $\gamma(t,M)=\pm t^kM^l$. Hence $P_0(t,M)=\pm t^kM^l P_1(t,t^{-4}M^{-1}).$

The equation $\alpha_{\BJ_K} \BJ_K=0$ can now be written as 
$$\BJ_K(n+1)=\pm \frac{ t^{2nl+k} P_1(t,t^{-4-2n})}{P_1(t,t^{2n})} \, \BJ_K(n).$$
This implies that $$br[\BJ_K(n+1)]-br[\BJ_K(n)]=br(t^{2nl+k} P_1(t,t^{-4-2n}))-br(P_1(t,t^{2n}).$$ It is easy to see that for $n$ big enough, $br(t^{2nl+k} P_1(t,t^{-4-2n}))-br(P_1(t,t^{2n}))$ is a constant independent of $n$. Hence the breadth of $\BJ_K(n)$, for $n$ big enough, is a linear function on $n$. This contradicts Lemma \ref{deg}, since $K$ is a non-trivial alternating knot.
\end{proof}

Let $\ve$ be the map reducing $t=-1$.

\begin{proposition}
\label{L-1}
For any $P \in \CA_{\BJ_K}$, $\ve(P)$ is divisible by $L-1$.
\end{proposition}

\begin{proof}

The proof of Proposition \ref{L-1} is similar to that of \cite[Proposition 2.3]{Le06}, which makes use of the Melvin-Morton conjecture proved by Bar-Natan and Garoufalidis \cite{BG}.

It is known that for any knot $K$ (with framing zero), $J_K(n)/[n]$ is a Laurent polynomial in $t^4$. Moreover, the Melvin-Morton conjecture \cite{MM} says that for any $z \in \BC^*$ we have $$\lim_{n \to \infty}  \left( \frac{J_K(n)}{[n]} \mid_{t^2=z^{1/n}} \right)=\frac{1}{\Delta_K(z)},$$ where $\Delta_K(z)$ is the Alexander polynomial of $K$.

For $l \in \BZ$ and $z \in \BC \setminus \{0,\pm 1\}$, we let 
\begin{eqnarray*}
\widehat{\BJ}_K(l, z) &:=& \lim_{n \to \infty}  \left( \frac{J_K(2n+2l+1)}{[2n+2l+1]} \mid_{t^2=z^{1/(2n+1)}} \right)\\
&=& \lim_{n \to \infty}  \left( \frac{t^2-t^{-2}}{z-z^{-1}} \, \BJ_K(n+l) \mid_{t^2=z^{1/(2n+1)}} \right) .
\end{eqnarray*}
Then $$\widehat{\BJ}_K(0,z)=\lim_{n \to \infty}  \left( \frac{J_K(2n+1)}{[2n+1]} \mid_{t^2=z^{1/(2n+1)}} \right)=\frac{1}{\Delta_K(z)}.$$ In particular, we have $\widehat{\BJ}_K(0,z) \not= 0$.

\smallskip

\textbf{Claim 1.} For any $l \in \BZ$, we have $\widehat{\BJ}_K(l, z)=\widehat{\BJ}_K(0,z)$.

\smallskip

\textit{Proof of Claim 1.} For any knot $K$, by \cite{MM} we have
$$ \frac{J_K(n)}{[n]}|_{t^4= e^h}= \sum_{k=0}^\infty P_{k}(n) h^k,$$
where $P_k(n)$ is a polynomial in $n$ of degree at most $k$:
$$P_k(n) = P_{k,k} n^k + P_{k,k-1} n^{k-1} +\dots P_{k,1} n +
P_{k,0}.$$
Then 
\begin{eqnarray*}
\widehat{\BJ}_K(l, z) &=& \lim_{n \to \infty}  \left( \frac{J_K(2n+2l+1)}{[2n+2l+1]} \mid_{t^2=z^{1/(2n+1)}} \right)\\
&=& \lim_{n \to \infty}  \left( \sum_{k=0}^\infty \sum_{j=0}^k P_{k,j} \, (2n+2l+1)^j h^k \mid_{h=\frac{2 \ln z}{2n+1}} \right).
\end{eqnarray*}

We have $$\lim_{n \to \infty}(2n+2l+1)^j \left( \frac{2 \ln z}{2n+1} \right)^k  = \begin{cases}
 0 & \text{if $j<k$}\\
 (2\ln z)^k & \text{if $j=k$}
 \end{cases},$$
 which is independent of $l$. Claim 1 follows.

\smallskip

We now complete the proof of Proposition \ref{L-1}. Suppose $P=\sum \lambda_{k,l} M^kL^l$, where $\lambda_{k,l} \in \CR$. Then $
\sum \lambda_{k,l} \, t^{2kn} \BJ_{K}(n+l)=0
$ for all integers $n$. 

For $z \in \BC \setminus \{0,\pm 1\}$, by Claim 1 we have
\begin{eqnarray*}
0 &=& \lim_{n \to \infty} \left( \sum \lambda_{k,l} \, t^{2kn} \frac{t^2-t^{-2}}{z-z^{-1}} \, \BJ_K(n+l) \mid_{t^2=z^{1/(2n+1)}} \right) \\
  &=& \sum (\lambda_{k,l} \mid_{t^2=1}) z^{k/2} \, \widehat{\BJ}_K(l,z) \\
  &=& (P \mid_{t^2=1,M=z^{1/2},L=1}) \widehat{\BJ}_K(0,z).
\end{eqnarray*}
Since $\widehat{\BJ}_K(0,z) \not= 0$, we have $P \mid_{t^2=1,M=z^{1/2},L=1}=0$ for all $z \in \BC \setminus \{0,\pm 1\}$. This implies that $P \mid_{t^2=1}$ is divisible by $L-1$. Proposition \ref{L-1} follows.
\end{proof}

\begin{proposition}
\label{l-1}
$\ve(\alpha_{\BJ_K})$ has $L$-degree 1 if and only if $\alpha_{\BJ_K}$ has $L$-degree 1.
\end{proposition}

\begin{proof}
The backward direction is obvious since $\varepsilon(\alpha_{\BJ_K})$ is always divisible by $L-1$, by Proposition \ref{L-1}. Suppose that $\varepsilon(\alpha_{\BJ_K})=g(M)(L-1)$ for some  $g(M) \in \BC[M^{\pm 1}] \setminus \{0\}.$	Then
	\begin{equation}
		\alpha_{\BJ_K}= g(M)(L-1)+(1+t) \sum_{k=0}^{d} a_k(M)L^k,
		\label{eq23}
	\end{equation}
	where $a_k(M) \in \CR[M^{\pm 1}]$ and $d$ is the $L$-degree of $\alpha_{\BJ_K}.$
	
	Since $\alpha_{\BJ_K}(t,(t^2M)^{-1},L^{-1})$ is also in the recurrence ideal of $\BJ_K$,  $$\alpha_{\BJ_K}(t,M,L) = h(M) \alpha_{\BJ_K}(t,(t^2M)^{-1},L^{-1}) L^{d}$$ for some $h(M) \in \CR(M)$. Eq. \eqref{eq23} then becomes
	\begin{eqnarray*}
	&& g(M)(L-1)+(1+t) \sum_{k=0}^{d}a_k(M)L^k \\
	&=& h(M)g(t^{-2}M^{-1})(L^{-1}-1)L^{d}+(1+t) \sum_{k=0}^{d} h(M)a_k(t^{-2}M^{-1})L^{d-k}. \end{eqnarray*}

	Suppose that $d>1$. By comparing the coefficients of $L^0$ in both sides of the above equation, we get $-g(M)+(1+t)a_0(M)=(1+t)h(M)a_{d}(t^{-2}M^{-1}).$ This is equivalent to
	\begin{equation}
	g(M)=(1+t)\left( a_0(M)-h(M)a_{d}(t^{-2}M^{-1}) \right).
	\label{eq24}
	\end{equation}
	 Since $g(M)$ is a Laurent polynomial in $M$ with coefficients in $\BC$, Eq. \eqref{eq24} implies that $g(M)=0$. This is a contradiction. Hence $d=1$.
\end{proof}

\section{Proof of Theorem \ref{main}}

\label{figure 8}

Let $E$ be the figure eight knot. By \cite{Habiro} we have
\begin{equation}
\label{Ha}
J_E(n)= [n] \sum_{k=0}^{n-1} \prod_{l=1}^k 
(t^{4n}+t^{-4n}-t^{4l}-t^{-4l}).
\end{equation}

Recall that $E^{(r,2)}$ is the $(r,2)$-cable of $E$ and $\BJ_E(n)=J_E(2n+1)$. By Lemma \ref{J_C}, we have
\begin{equation}
\label{f}
M^r(L+t^{-2r}M^{-2r})J_{E^{(r,2)}} = \BJ_E.
\end{equation} 

For non-zero $f,g \in \BC[M^{\pm 1},L]$, we write $f \eqM g$ if the quotient $f/g$ does not depend on $L$. Proving Theorem \ref{main} is then equivalent to proving that $\ve(\alpha_{E^{(r,2)}}) \eqM A_{E^{(r,2)}}$,
where $A_{E^{(r,2)}}=$ $$(L-1) \left\{ L^2 - ( (M^8+M^{-8}-M^4-M^{-4}-2)^2 -2 ) L+1 \right\} (L+M^{-2r})$$ is the A-polynomial of $E^{(r,2)}$ c.f. \cite{NZ}. 

The proof of $\ve(\alpha_{E^{(r,2)}}) \eqM A_{E^{(r,2)}}$ is divided into 4 steps.

\subsection{Degree formulas for the colored Jones polynomials} The following lemma will be  used later in the proof of Theorem \ref{main}.

\label{degr}

\begin{lemma}
\label{degree}
For $n>0$ we have 
\begin{eqnarray*}
d_+[J_E(n)] &=&  4n^2-2n-2,\\
d_-[J_E(n)] &=& -4n^2+2n+2,\\
d_+[J_{E^{(r,2)}}(n)] &=& \begin{cases} 16n^2-(2r+20)n+2r+4 &\mbox{if } r \ge -7 \\ 
-2rn^2+2r & \mbox{if } r \le -9, \end{cases} \\
d_-[J_{E^{(r,2)}}(n)] &=& \begin{cases} -2rn^2+2r &\mbox{if } r \ge 9 \\ 
-16n^2-(2r-20)n+2r-4 & \mbox{if } r \le 7. \end{cases}
\end{eqnarray*}
\end{lemma}

\begin{proof} The first two formulas follow directly from Eq. \eqref{Ha}. We now prove the formula for $d_+[J_{E^{(r,2)}}(n)]$. The one for $d_-[J_{E^{(r,2)}}(n)]$ is proved similarly. 

From Eq. \eqref{cables}, we have
\begin{eqnarray*}
d_+[J_{E^{(r,2)}}(n)] &=& -2r(n^2-1) + \max_{1 \le k \le n} \{ 2rk(k-1)+ d_+[J_E(2k-1)]\}\\
                  &=& -2r(n^2-1) + \max_{1 \le k \le n} \{ (2r+16)k^2-(2r+20)k+4\}.
\end{eqnarray*}
Let $f(k):=(2r+16)k^2-(2r+20)k+4$, where $1 \le k \le n$. If $r \ge -7$, $f(k)$ attains its maximum at $k=n$. If $r \le -9$, $f(k)$ attains its maximum at $k=1$. The lemma follows.
\end{proof}

\subsection{An inhomogeneous recurrence relation for $\BJ_E$} Let 
\begin{eqnarray*}
P_1(t,M) &:=& t^{-2}M^2-t^2M^{-2}, \\
P_{-1}(t,M) &:=& t^2M^2-t^{-2}M^{-2}, \\ 
P_0(t,M) &:=& (M^2-M^{-2})(-M^4-M^{-4}+M^2+M^{-2}+t^4+t^{-4}).
\end{eqnarray*}
From \cite[Proposition 4.4]{CM} (see also \cite{GS}) we have 
\begin{equation}
\label{CM}
(P_1L+P_{-1}L^{-1}+P_0)J_E \in \CR[M^{\pm 1}].
\end{equation}

Let 
\begin{eqnarray*} 
Q_1(t,M) &:=& P_1(t,M)P_1(t,t^2M)P_0(t,t^{-2}M),\\
Q_{-1}(t,M) &:=& P_{-1}(t,M)P_{-1}(t,t^{-2}M)P_0(t,t^{2}M),\\
Q_0(t,M) &:=& P_1(t,M)P_{-1}(t,t^{2}M)P_0(t,t^{-2}M) + P_{-1}(t,M)P_{1}(t,t^{-2}M)P_0(t,t^2M)\\
         && - \, P_0(t,M)P_0(t,t^2M)P_0(t,t^{-2}M).
\end{eqnarray*} 

\begin{proposition}
\label{ff}
We have $$\left\{ Q_1(t,t^2M^2) L + Q_{-1}(t,t^2M^2)  L^{-1}+Q_0(t,t^2M^2) \right\}\BJ_E \in \CR[M^{\pm 1}].$$
\end{proposition}

\begin{proof}
We first note that
\begin{eqnarray*}
&& Q_1(t,M) L^2+Q_{-1}(t,M) L^{-2}+Q_0(t,M) \\
 &=& P_1(t,M)P_1(t,t^2M)P_0(t,t^{-2}M) L^2 + P_{-1}(t,M)P_{-1}(t,t^{-2}M)P_0(t,t^{2}M) L^{-2} \\
&& + \, P_1(t,M)P_{-1}(t,t^{2}M)P_0(t,t^{-2}M) + P_{-1}(t,M)P_{1}(t,t^{-2}M)P_0(t,t^2M)\\
         && - \, P_0(t,M)P_0(t,t^2M)P_0(t,t^{-2}M)\\
&=& \left\{ P_1(t,M)P_0(t,t^{-2}M) L + P_{-1}(t,M) P_0(t,t^2M) L^{-1} - P_0(t,t^2M) P_0(t,t^{-2}M) \right\} \\
&& \times \left\{ P_1(t,M) L+P_{-1}(t,M) L^{-1}+P_0(t,M) \right\}.
\end{eqnarray*} 
By Eq. \eqref{CM} we have $(P_1L+P_{-1}L^{-1}+P_0)J_E \in \CR[M^{\pm 1}]$. Hence 
\begin{equation} 
\label{abc}
(Q_1 L^2 + Q_{-1} L^{-2}+Q_0)J_E \in \CR[M^{\pm 1}].
\end{equation}

We have $(M^k L^{2l}J_E)(2n+1)=((t^2M^2)^kL^{l}\BJ_E)(n)$. It follows that $$(P(t,M)L^{2l}J_E)(2n+1)=(P(t,t^2M^2)L^{l}\BJ_E)(n)$$ for any $P(t,M) \in \CR[M^{\pm 1}].$ Hence Eq. \eqref{abc} implies that
$$\left\{ Q_1(t,t^2M^2) L + Q_{-1}(t,t^2M^2)  L^{-1}+Q_0(t,t^2M^2) \right\}\BJ_E \in \CR[M^{\pm 1}].$$
This proves Proposition \ref{ff}.
\end{proof}

\subsection{A recurrence relation for $J_{E^{(r,2)}}$}

Let $$Q(t,M, L):=Q_1(t,t^2M^2) L + Q_{-1}(t,t^2M^2)  L^{-1}+Q_0(t,t^2M^2).$$ By Proposition \ref{ff}, we have $Q\BJ_E \in \CR[M^{\pm 1}].$  Eq. \eqref{f} then implies that 
\begin{equation}
\label{bap}
QM^r(L+t^{-2r}M^{-2r})J_{E^{(r,2)}} \in  \CR[M^{\pm 1}].
\end{equation}

Let $Q'(t,M):=LQ(t,M)M^r(L+t^{-2r}M^{-2r})$. From Eq. \eqref{bap} we have $Q'J_{E^{(r,2)}} \in  \CR[M^{\pm 1}]$. 

Let $R:=Q'J_{E^{(r,2)}} \in  \CR[M^{\pm 1}]$. We claim that $R \not= 0$, which means that $Q'J_{E^{(r,2)}}=R$ is an inhomogeneous recurrence relation for $J_{E^{(r,2)}}$. Indeed, assume that $R=0$. Then $Q'$ annihilates the colored Jones function $J_{E^{(r,2)}}$. By \cite[Proposition 2.3]{Le06}, $\ve(Q')$ is divisible by $L-1$. However this cannot occur, since 
$$
\ve(Q') \eqM \left\{ L^2 - \left( (M^8+M^{-8}-M^4-M^{-4}-2)^2 -2 \right) L+1 \right\} (L+M^{-2r})
$$
is not divisible by $L-1$. Hence $R \not= 0$ in $\CR[M^{\pm 1}]$.

Write $R(t,M)=(1+t)^m R'(t,M)$, where $m \ge 0$ and $R'(-1,M) \not= 0$ in $\BC[M^{\pm 1}]$. Let $$S(t,M,L):=(R'(t,M)L-R'(t,t^2M))Q'(t,M).$$ Since $Q'J_{E^{(r,2)}}=(1+t)^m R' \in  \CR[M^{\pm 1}]$ is an inhomogeneous recurrence relation for $J_{E^{(r,2)}}$, we have the following.

\begin{proposition}
The polynomial $S \in \CT$ annihilates the colored Jones function $J_{E^{(r,2)}}$ and has $L$-degree 4.
\end{proposition}

\subsection{Completing the proof of Theorem \ref{main}} Note that $S$ has $L$-degree 4 and $\ve(S) \eqM A_{E^{(r,2)}}.$ Hence to complete the proof of Theorem \ref{main}, we only need to show that if $|r| \ge 9$ then $S$ is equal to the recurrence polynomial $\alpha_{E^{(r,2)}}$ in $\tilde{\CT}$, up to a rational function in $\CR(M)$. This is achieved by showing that there does not exist a non-zero polynomial $P \in \CR[M^{\pm 1}][L]$ of degree $\le 3$ that annihilates the colored Jones function $J_{E^{(r,2)}}$. We will make use of the degree formulas in Subsection \ref{degr}. 

From now on we assume that $r$ is an odd integer satisfying $|r| \ge 9$. Suppose that $P=P_3L^3+P_2L^2+P_1L+P_0$, where $P_k \in \CR[M^{\pm 1}]$, annihilates $J_{E^{(r,2)}}$. We want to show that $P_k=0$ for $0 \le k \le 3$. 

Indeed, by applying Lemma \ref{J_C} we have 
\begin{eqnarray*}
0 &=& P_3 J_{E^{(r,2)}}(n+3) + P_2 J_{E^{(r,2)}}(n+2) + P_1 J_{E^{(r,2)}}(n+1) + P_0 J_{E^{(r,2)}}(n)\\
  &=& \left( -t^{-2r(6n+9)}P_3+t^{-2r(4n+4)}P_2-t^{-2r(2n+1)}P_1+P_0 \right) J_{E^{(r,2)}}(n) \\
  && + \left( t^{-2r(5n+8)}P_3-t^{-2r(3n+3)}P_2+t^{-2rn}P_1 \right) J_E(2n+1) \\
  && + \left( -t^{-2r(3n+6)}P_3+t^{-2r(n+1)}P_2 \right) J_E(2n+3)+ t^{-2r(n+2)}P_3 J_E(2n+5)\\
  &=& P'_3J_{E^{(r,2)}}(n) + P'_2 J_E(2n+5) + P'_1 J_E(2n+3) +P'_0 J_E(2n+1).
\end{eqnarray*}
It is easy to see that $P_k=0$ for $0 \le k \le 3$ if and only if $P'_k=0$ for $0 \le k \le 3$. 

Let $g(n)=P'_2 J_E(2n+5) + P'_1 J_E(2n+3) +P'_0 J_E(2n+1)$. Then 
\begin{equation}
\label{g(n)}
P'_3J_{E^{(r,2)}}(n)+g(n)=0.
\end{equation}

We first show that $P'_3=0$. Indeed, assume that $P'_3 \not=0$ in $\CR[M^{\pm 1}]$. If $r \ge 9$ then, by Lemma \ref{degree}, we have $$d_-[P'_3 J_{E^{(r,2)}}(n)] =d_-[ J_{E^{(r,2)}}(n)]+O(n)=-2rn^2 + O(n).$$ 
Similarly, we have $d_-[P'_k J_E(2n+2k+1)] = -16n^2+O(n)$ if $P'_k \not= 0$, where $k=0,1,2$.
It follows that, for $n$ big enough,
\begin{eqnarray*}
d_-[P'_3 J_{E^{(r,2)}}(n)] &<& \min \{ d_-[P'_2 J_E(2n+5)], d_-[P'_1 J_E(2n+3)], d_-[P'_0 J_E(2n+1)]\}\\
&\le& d_-[g(n)].
\end{eqnarray*}
Hence $d_-[P'_3 J_{E^{(r,2)}}(n)] < d_-[g(n)].$ This contradicts Eq. \eqref{g(n)}.

If $r \le -9$ then, by similar arguments as above, we have
\begin{eqnarray*}
d_+[P'_3 J_{E^{(r,2)}}(n)] &>& \max \{ d_+[P'_2 J_E(2n+5)], d_+[P'_1 J_E(2n+3)], d_+[P'_0 J_E(2n+1)]\}\\
&\ge& d_+[g(n)].
\end{eqnarray*}
for $n$ big enough. This also contradicts Eq. \eqref{g(n)}. Hence $P'_3=0$.

Since $g(n)=0$, we have $(P'_2 L^2+P'_1 L + P'_0)\BJ_E=0$. This means that $\BJ_E$ is annihilated by $P':=P'_2 L^2+P'_1 L + P'_0$. We claim that $P'=0$ in $\CR[M^{\pm 1}][L]$. Indeed, assume that $P' \not=0$. Since $P'$ annihilates $\BJ_E$, it is divisible by the recurrence polynomial $\alpha_{\BJ_E}$ in $\tilde{\CT}$. It follows that $\alpha_{\BJ_E}$, and hence $\ve(\alpha_{\BJ_E})$, has $L$-degree $\le 2$.

Since $E$ is a non-trivial alternating knot, Propositions \ref{alt}, \ref{L-1} and \ref{l-1}  imply that $\ve(\alpha_{\BJ_E})$ is divisible by $L-1$ and has $L$-degree $\ge 2$. Hence we conclude that $\ve(\alpha_{\BJ_E})$ is divisible by $L-1$ and has $L$-degree exactly 2.

By Proposition \ref{ff}, we have $Q\BJ_E \in \CR[M^{\pm 1}]$. Let $Q'':=Q\BJ_E$. Then $Q'' \not= 0$ (otherwise, $Q$ annihilates $\BJ_E$. However, this contradicts Proposition \ref{L-1} since $\ve(Q) \eqM L^2 - \left( (M^8+M^{-8}-M^4-M^{-4}-2)^2 -2 \right) L+1$ is not divisible by $L-1$). This means that $Q\BJ_E=Q''  \in \CR[M^{\pm 1}]$ is an inhomogeneous recurrence relation for $\BJ_E$

Write $Q''(t,M)=(1+t)^m Q'''(t,M)$, where $m \ge 0$ and $Q'''(-1,M) \not= 0$ in $\BC[M^{\pm 1}]$. Then $(Q'''(t,M)L-Q'''(t,t^2M))Q$ annihilates $\BJ_E$ and hence is divisible by $\alpha_{\BJ_E}$ in $\tilde{\CT}$. Consequently, $(L-1)\ve(Q)$ is divisible by $\ve(\alpha_{\BJ_E})$ in $\BC(M)[L]$. This means $\frac{\ve(\alpha_{\BJ_E})}{L-1}$ divides $\ve(Q)$ in $\BC(M)[L]$. However this cannot occur, since $\frac{\ve(\alpha_{\BJ_E})}{L-1}$ has $L$-degree exactly 1 and $\ve(Q)$ is an irreducible polynomial in $\BC[M^{\pm 1}, L]$ of $L$-degree 2. 

Hence $P'=0$, which means that $P'_k=0$ for $0 \le k \le 2$. Consequently, $P_k=0$ for $0 \le k \le 3$. This completes the proof of Theorem \ref{main}.

\end{document}